\documentclass[11pt]{amsart}
\usepackage{amsmath,amstext,amssymb,amsopn,amsthm}
\usepackage{url}
\usepackage[unicode,bookmarks,colorlinks]{hyperref}

\newtheorem{thm}{Theorem}[section]
\newtheorem{conjecture}[thm]{Conjecture}
\newtheorem{corollary}[thm]{Corollary}

\newtheorem{remark}[thm]{Remark}

\newtheorem{proposition}[thm]{Proposition}

\newtheorem{prop}[thm]{Proposition}
\numberwithin{equation}{section}

 \newtheorem{lem}[thm]{Lemma}

 \newtheorem{defn}[thm]{Definition}

 \newcommand{\Real}{{\mathbb{R}}}
 \newcommand{\N}{{\mathbb{N}}}
 \newcommand{\Z}{{\mathbb{Z}}}
 \newcommand{\Rd}{{\Real^d}}
 \renewcommand{\P}{{\bf P}}
 \newcommand{\E}{{\bf E}}

\newcommand{\bB}{\mathrm{I\! B\!}}

\newcommand{\bR}{\mathrm{I\! R\!}}

\title[Heat kernel monotonicity]{Neumann Heat kernel 
monotonicity}
\author[Ba\~nuelos,  Kulczycki,   Siudeja]
{R{.} Ba\~nuelos*,  T{.} Kulczycki** and B{.} Siudeja*}
\address{*Department of Mathematics,
Purdue University,
West Lafayette, IN 47907-1395\newline
**Institute of Mathematics,
 Wroc{\l}aw University of Technology,
50-370 Wroc{\l}aw, Poland}
\email{Ba\~nuelos: banuelos@math.purdue.edu}
\email{Kulczycki: Tadeusz.Kulczycki@pwr.wroc.pl}
\email{Siudeja: siudeja@math.purdue.edu}
\begin{document}

\sloppy
\footnotetext{
\emph{The first and third authors were supported in part by NSF grant
\#  0603701-DMS.}\\
 \emph{The second author was supported in part by KBN Grant 1 P03A 020 28}\\ 
\emph{2000 Mathematics Subject Classification:} Primary 31B05, 60J45.\\
\emph{Key words and phrases:} Neumann Heat kernels, 
reflected Brownian motion, random walks, Bessel processes.}
\begin{abstract}
We prove that the diagonal of the transition probabilities for the 
$d$--dimensional Bessel processes on $(0, 1]$, reflected at 
$1$,  which we denote by $p_R^N(t, r,r)$,  is an increasing function of $r$ for $d>2$ and that this is false for $d=2$. 
\end{abstract}

\maketitle
\section{Introduction}

The following conjecture of Richard Laugesen and Carlo Morpurgo arose, as communicated  to us
 by R. Laugesen, in connection with their work in \cite{Lamo} on conformal extremals  of  zeta
functions of eigenvalues  under Neumann boundary conditions.  While this may be the first time the conjecture appears in print, the problem seems to be well--known. 

\begin{conjecture}
\label{conj1}
Let $\bB$ be the unit ball in $\bR^d$, $d\geq 2$, and 
let 
$p^N_{\bB}(t, x, y)$ be the heat kernel for the Laplacian in $\bB$ 
with Neumann boundary conditions.  Equivalently,  $p^N_{\bB}(t, x, y)$ 
gives the transition probabilities  for the Brownian motion in 
$\bB$ with normal reflection on the boundary.    Fix $t>0$.  The (radial) function 
$p^N_{\bB}(t, x, x)$  increases  as $|x|$ increases  to 1.  That is, for all $t>0$, 
\begin{equation} p^N_{\bB}(t, x_1, x_1)< p^N_{\bB}(t, x_2, x_2),
\end{equation} 
whenever  $0\leq \left| x_1\right| <\left| x_2\right|\leq 1$. 
\end{conjecture}  

 Of course, the same conjecture makes sense for $d=1$. 
 For this, see Remark \ref{onedim0}  in \S5.

 We should observe here that for the Dirichlet heat kernel in
$\bB$\,, the opposite inequality is true.  That is, the diagonal of
the Dirichlet heat kernel decreases as the point moves toward
the boundary (see Proposition \ref{dirichlet1} in \S5 below).

The conjecture is closely related to the {\it hot--spots
Conjecture} of Jeff Rauch which asserts that the maxima and
minima of any eigenfunction $\varphi _1$ corresponding to the smallest positive 
Neumann eigenvalue $\mu_1$ of a convex planar domain are attained at
the boundary, and only at the boundary, of the domain.
Indeed, if we denote the volume of the unit ball $\bB$ in
$\bR^d$ by $\omega_d$, the  eigenfunction expansion of the heat kernels 
gives that 
\begin{equation}
p^N_{\bB}(t, x, x) \approx \frac 1{\omega_d} +e^{-\mu _1 t}|\varphi _1(x)|^2,
\end{equation}
and this is uniform in $x$ for $t$ large (see \cite{smits}).  We refer the reader to \cite{BaBu}, \cite{Bu},
\cite{BaPanPas}, and references therein, for more on the {\it
hot--spots} conjecture and for the use of heat kernel expansions
and transition probabilities for that problem. Of course, for
the unit ball the {\it hot--spots} conjecture follows easily
from the explicit expression of $\varphi_1$ as a 
Bessel function.  However, a more general 
Laugesen-Morpurgo Conjecture can be stated 
where the connection to  the {\it hot--spots} conjecture is more
meaningful, see Conjecture \ref{conj2} below. Surprisingly the {\it hot--spots} conjecture is open
even for an arbitrary triangle in the plane. But perhaps even
more surprising is the fact that the Laugesen--Morpurgo
conjecture   is open even for the unit disk in the plane. 

The Neumann heat kernel $p^N_{\bB}(t, x, y)$ gives the
transition probabilities for the Brownian motion reflected on
the boundary of the ball and hence the use of probability for this problem 
(just as in the case of the {\it hot--spots} conjecture) is very natural. 
The Brownian motion in the ball has a
skew symmetric decomposition in terms of a Bessel processes
(the radial part) and spherical Brownian motion running with
a clock that depends on the Bessel processes (see for example \cite{BaSm}). That is,
let $W_t^{\bB}$ be reflected $d$--dimensional Brownian
motion (RBM) in the ball $\bB$ and let $R_t$ be the $d$--dimensional
Bessel process in the interval $I=(0,1]$ reflected at $1$. Then
$R_t$ is the radial part of $W_t^{\bB}$. That is,
$R_t=|W_t^{\bB}|$. Let $p^{R}_I(t, r, \rho)$ be the transition
probabilities for $R_t$ in the interval $I=(0, 1]$. We will
often refer to this as the heat kernel for $|W_t^{\bB}|$. 
The main result of this paper is the following 

\begin{thm}\label{main}  Suppose $d>2$. Fix $t>0$.  
The function $p^{R}_I(t, r, r)$ is increasing in $r$.  That is, 
\begin{equation}\label{radial1}
p^{R}_I(t, r_1, r_1)<p^{R}_I(t, r_2, r_2)
\end{equation}
for $0<r_1<r_2\leq 1$.  This monotonicity property  fails  if  $d=2$.
\end{thm}


It is known (see \cite{RY}, page 415) that the transition probabilities 
(heat kernel) $q(t,r,\rho)$ for the free
$2$--dimensional Bessel process is given by
\begin{gather}\label{2ddens}
  q(t,r,\rho)=t^{-1} \rho e^{-{r^2+\rho^2\over 2t}}I_0\left( r\rho\over t \right),
\end{gather}
where $I_0$ is the modified Bessel function of order $0$. The function $q(t,r,r)$
is not increasing.  In fact,  as $t\to 0$ this function has a ``tall bump" moving 
toward $0$.
Since the reflected process is equal to the free process before the first reflection,
it seems reasonable to expect that $p^{R}_I(t, r, r)$ is not increasing for small values of $t$ when $d=2$.  A rigorous proof of this fact will be given below. 
On the other hand, the function $q(t,r,r)$ is non-decreasing 
for 
$d$--dimensional Bessel processes when $d\geq 3$ and therefore 
one may expect the monotonicity property to hold for 
$p^{R}_I(t, r, r)$ for such $d$ and  Theorem \ref{main} shows that 
this is indeed the case. 

Our strategy  in this paper is to replace the reflected
Bessel process by a random walk and obtain the result for
this  walk. We then show that under the appropriate scale 
the random walk converges to the reflected Bessel process.  

The
paper is organized as follows. In \S 2 we introduce the random
walks and prove the analogue of Theorem \ref{main} for these.
 \S 3 contains the
proof of the convergence of these random walks to their
continuous counterparts and \S 4 gives the proof of  theorem \ref{main}. 
Finally, the last section contains some conjectures related to our result and illuminates the connection between the Rauch {\it hot--spots} Conjecture and the Laugesen--Morpurgo Conjecture further.

\section{The random walk} 
In this section we introduce the random walk which we will use later to approximate
the reflected Bessel process.
We will use the following notation. For any $x\in\Rd$, $|x|$ denote the length of the vector $x$.
For $d\geq3$ and $x\not = 0$,  set $$U(x)={|x|+1\over|x|}x$$ and $$D(x)={|x|-1\over|x|}x.$$  Consider the two sets
\begin{equation*} 
   C(x)=\left\{ y\in\Rd: |y|=|x|+1\mbox{ and }y-D(x)\bot x\right\}, 
  \end{equation*}
  and 
  \begin{equation*}
 S=\left\{ y\in\Rd: |y|\in \N\mbox{ and }|y|\geq {d\over2}-1\right\}.
\end{equation*}
Note that $C(x)$ is
a $(d-2)$--dimensional sphere with center at $D(x)$ and 
orthogonal to $x$. We now define two random walks as follows.  
\begin{defn}
  Let $X_n$ be a random walk on $S$ with the following transition probabilities  
  \begin{enumerate}
    \item $p(1,x,U(x))=\frac12$,\\
    \item $p(1,x,D(x))=\frac12-{d-1\over 4|x|}$, for $|x|\geq {d-1\over2}$,\\
    \item $p(1,x,A)={d-1\over 4|x|}\mu_x(A)$, where $A\subset C(x)$ and $\mu_x$ a
      uniform probability measure on $C(x)$, for $|x|\geq {d-1\over 2}$,\\  
    \item $p(1,x,A)=\frac12\mu_x(A)$, where $A\subset C(x)$ and $\mu_x$ a
      uniform probability measure on $C(x)$, for even $d$ and $|x|={d-2\over2}$.
  \end{enumerate}
\end{defn}

Observe that if $d=2k+1$, then $|x|\geq k$ and $p(1,x,D(x))=0$ if $|x|=k$.
If $d=2k$, then $|x|\geq k-1$ and $p(1,x,D(x))={1\over 4k}$ if $|x|=k$. 
Hence, we need the 
additional points $|x|=k-1$ as in  the last 
case in the definition.

We will also be concerned with the radial part of the above random walk. 
\begin{defn}
\label{defY_n}
  Let $Y_n=|X_n|$. Then  $Y_n$
 has the following transition 
  probabilities
  \begin{enumerate}
    \item $p(1,m,m+1)=\frac12+{d-1\over 4m}$, for $m\geq {d-1\over2}$,\\
    \item $p(1,m,m+1)=1$, for even $d$ and $m={d\over2}-1$,\\
    \item $p(1,m,m-1)=\frac12-{d-1\over 4m}$, for $m\geq {d-1\over2}$.
  \end{enumerate}
\end{defn}

The above are the ``free" random walks.  For our purpose,  we define the reflected versions of these walks.

\begin{defn}
\label{defY_n^N}
  Let $X^N_n$ be the random walk on $S\cap \left\{ |x|\leq N \right\}$ with
  transition probabilities $p_N(1,x,y)=p(1,x,y)$ for $|x|<N$ and
  \begin{enumerate}
    \item $p_N(1,x,x)=\frac12+{d-1\over 4|x|}$ for $|x|=N$,\\
    \item $p_N(1,x,D(x))=\frac12 -{d-1\over 4|x|}$ for $|x|=N$
  \end{enumerate}
  and denote by $Y^N_n=|X^N_n|$ its radial part. 
\end{defn}

We will use $X^N_n$ and $Y^N_n$ to approximate 
$W_t^{\bB}$ and $R_t=|W_t^{\bB}|$, respectively. 
We will now show that $p_N(n,m,m)$, the transition
probabilities for the random walk $Y^N_n$, are increasing
with $m$ for any fixed $n$. The following property will be
useful for this purpose. 

\begin{defn}
  We say, that a random walk has the nondecreasing loop property if for any $m$ its transition probabilities have the property that 
  $$p(1,m,m+1)p(1,m+1,m)$$ is nondecreasing with $m$. In case of a reflected walk we also require
  that $p_N(1,N,N)^2$ (the loop at the reflection point) be larger than  or equal to 
  $$p(1,m,m+1)p(1,m+1,m),$$ for every $m<N$.
\end{defn}

\begin{lem}\label{loop}  The random walks 
  $Y_n$ and $Y^N_n$ have the nondecreasing loop property.
\end{lem}

\begin{proof}  In the case of  $Y_n$ we have 
for $m \ge (d - 1)/2$
  \begin{gather}
    \begin{split}
      p(1,m,m+1)p(1,m+1,m)&=
      \left(\frac12+{d-1\over4m} \right)\left(\frac12-{d-1\over4(m+1)}  \right)
      \\&=
      {(2m+d-1)(2m+2-d+1)\over16m(m+1)}
      \\&=
      {4(m^2+m)-(d-1)(d-3)\over16(m^2+m)}
      \\&=
      \frac14-{(d-1)(d-3)\over 16(m^2+m)}.
    \end{split}
  \end{gather}
Thus the left hand side  is 
nondecreasing for $m\geq {d-1\over2}$. When $d=3$ this quantity is constant.   

  Next we take $m={d\over2}-1$ for $d$ even. In this case
  \begin{gather}
    \begin{split}
      p\left(1,{d\over2}-1,{d\over2}\right)p\left(1,{d\over2},{d\over2}-1\right)=
      {1\over 2d}.
    \end{split}
  \end{gather}
  From the general case
  \begin{eqnarray*}
      p\left(1,{d\over2},{d\over2}+1\right)p\left(1,{d\over2}+1,{d\over2}\right)&=&
      \frac14-{(d-1)(d-3)\over4d^2+8d}\\
      &=&
      {6d-3\over 4d^2+8d}
      \\
      &>&
      {2d+4\over 4d^2+8d}={1\over2d},
   \end{eqnarray*}
   since $d\geq4$. This completes  the 
   proof that $Y_n$ has the non-decreasing loop property. 

   In the case of $Y^N_n$ we are left with 
   reflection point loop, i.e. $p(1,N,N)^2$. But this is larger  then 
   $1/4$,  hence we have a nondecreasing loop property for $Y^N_n$.
\end{proof}

\begin{prop}  Fix $n$.  Then $p_N(n , m, m)$ is increasing in $m$. 
\end{prop}
\begin{proof} 
To prove this  we fix $n$ and consider each possible 
path from $m$ to $m$ in $n$ steps. The proof will be completed  if for each of them we can find a unique
path from $m+1$ to $m+1$ in $n$ steps that has a larger probability.
Towards this end, let $P_m=\{m=l_1,l_2,\cdots,l_{n-1},l_n=m\}$ be a path for $Y^N$. Let 
$k_1=\inf\{k:l_k=N\}$ and $k_2=\sup\{k:l_k=N\}$. 
If this path never touches $N$ then we can take the path
$\P_{m+1}=\{m+1=l_1+1,l_2+1,\cdots,l_n+1=m+1\}$. Since both paths start and end 
at the 
same point, they are (possibly after rearranging) a sequences of loops. 
By the nondecreasing loop property for $Y^N$ proved in Lemma \ref{loop},
the probability for the 
path $P_{m+1}$ is larger.

Up to now we have only used those paths starting at $m+1$ that do not remain at $N$.  
That is, those paths which move to $N-1$ immediately after hitting $N$. This is true,
since all the paths $P_m$ we considered above never touched $N$, their maximum
can be $N-1$ and the walk has to move to $N-2$ from there.

Suppose that $P_m$ hits $N$ at time $k_1<N$. Then $k_2<N$. 
Let $P_{m+1}=\{m+1=l_1+1,l_2+1,\cdots,l_{k_1-1}+1,l_{k_1},\cdots,l_{k_2},
l_{k_2+1}+1,\cdots,
l_n + 1 = m + 1\}$. The idea is to shift the parts of the path before
and after hitting $N$ for the first and last time. We have to show that such correspondence of 
the paths is one-to-one.

We have to look at the parts of $P_m$ before, after and in between the hitting times,
separately. First notice that for $P_{m+1}$ the step from $k_1-1$ to $k_1$ is the first
time the path remains at $N$. Similarly $k_2$ to $k_2+1$ is the last time the path
remains in $N$. Hence if two different paths $P_m$ have  different times $k_1$ or $k_2$,
then the shifted paths $P_{m+1}$ will also be different. Hence the only possible
paths with the same corresponding shifted paths must have the same hitting times
$k_1$ and $k_2$. 

Note that it is important in this proof that the walk cannot stay at any point
other then the reflection point $N$. If we allow $p(1,N-1,N-1)$ to be non-zero, then
the path $P_{m+1}$ obtained from the path $P_m$ that never touches $N$ 
may remain at the point $N$ even before the time $k_1$. This would invalidate the
above reasoning.

If two paths  $P_m$ are different at any shifted point (before $k_1$ or after $k_2$, 
then the same is true for the corresponding paths $P_{m+1}$.
Finally, if both $k_1$ and $k_2$ are the same for two different paths $P_m$ and they are the same
before $k_1$ and after $k_2$, then there must be a difference between $k_1$ and $k_2$.
But this part of those paths is not changed in $P_{m+1}$. Therefore the correspondence
between $P_m$ and $P_{m+1}$ is one-to-one. 

The last thing to check is that the probability of the corresponding $P_{m+1}$ path is larger.
Since the part between $k_1$ and $k_2$ is exactly the same, we can disregard it.
What is left is just a sequence of loops (after rearrangement). Hence by the nondecreasing 
loop property (Lemma \ref{loop}), this completes the proof.
\end{proof}

As a corollary to the above argument we  get
\begin{corollary}\label{trans}
  Fix $n$. For any $m'<m$,
  $p > 0$ and $m + p \le N$ we have $p_N(n,m,m')\leq p_N(n,m+p,m'+p)$.
\end{corollary}
The proof is almost identical. The only difference is that each path can be 
decomposed into loops plus additional transitions from $m$ to $m'$ (all of them toward $0$). 
But since $p_N(1,n,n-1)=\frac12-\frac{d-1}{4n}$ is increasing with $n$,
the path shifted by $p$ will have larger  probability.\qed

\section{Convergence}

\begin{prop}  The sequence $\{\frac{1}{N} X^N_{[N^2t]}\}$ converges weakly to the reflected Brownian motion $W^{\bB}_t$ as $N\to\infty$. 
\end{prop}
The proof of this fact is essentially the same as the convergence proof in \cite{D}. First
we need to establish the existence of a weak limit of the process $Z^N_t$ that 
interpolates $\frac{1}{N}X^N_{[N^2t]}$ linearly. That is, of the continuous process that equals  $\frac{1}{N}X^N_{[N^2t]}$ at the times of the jumps of the process and
is linear in between.
The process  $Z^N_t$ converges weakly by  H\"older 
continuity and Prohorov's theorem (see \cite{D}.) 
Our main goal here is to identify this weak limit as RBM on $\bB$. 
To accomplish this we use the submartingale characterization 
of the reflected Brownian
motion (see introduction in \cite{VS}).  More precisely, 
the RBM in $\bB$ is the only stochastic process starting from
$x\in \bB$ such that for any $f\in C^2_b(\bB)$ 
with positive normal derivative at each point of the boundary of $\bB$,  the process
\begin{gather}
  f(W^{\bB}_t)-\int_0^t \Delta f(W^{\bB}_s)ds
\end{gather}
is a submartingale.  Hence, to prove that $Z^N_t\to W^{\bB}_t$ weakly, 
 it is enough to show that
\begin{gather}\label{inequality}
  \liminf_{N\in {\mathbb N}} E\left( f(Z^N_t)-f(Z^N_s)-
  \frac12\int_s^t \Delta f(Z^N_u) du \right)
  \geq 0,
\end{gather}
for all such functions $f$.  Here and in the sequel, $\Delta$ denotes the Laplacian in $\bR^d$. 

First we will calculate an expectation of the 
single jump of $\frac1N X^N_{[N^2t]}$. Note that 
by the definition, $Z^N_t=\frac1N X^N_{[N^2t]}$ at the jump times.
\begin{lem}
  Let $u_n=n/N^2$ be the points where
   the process $\frac1N X^N_{[N^2t]}$ makes 
  its jumps. Then
  \begin{gather}
    \begin{split}
    \E^x\left( f(Z^N_{u_{n+1}})-f(Z^N_{u_n}) \right)&=
    \E^x\Bigl({1\over2N^2}\Delta f(Z^N_{u_n})+o(N^{-2})
    \\&+
    O(N^{-2})1_{\left\{|Z^N_{u_n}|={d-2\over2N}\right\}}
    \\&+
    (-c_{N}\partial_1 f(Z^N_{u_n})+O(N^{-2}))1_{\left\{|Z^N_{u_n}|=1\right\}}\Bigr),
  \end{split}
  \end{gather}
  where $\partial_1f$ denotes the outer normal 
derivative of $f$ on $\{y: |y|=|x|\}$. 
\end{lem}

\begin{proof}
  Let $A(x)=\E^{Nx}(f(\frac{1}{N} X^N_1)-f(x))$. Note that if the starting point for the process $Z^N_{u_n}$ is $x$, then the corresponding starting point of $X_n^N$ is $Nx$.
  By the strong Markov property for $X^N$ and the 
definition of $Z^N_t$,
\begin{gather}
\label{markov}
  \begin{split}
  \E^x\left( f(Z^N_{u_{n+1}})-f(Z^N_{u_n}) \right)&=
  \E^{Nx}\left( f\left(\frac{1}{N} X^N_{n+1}\right)-f\left(\frac{1}{N} X^N_n\right)\right)
  \\&=
  \E^{Nx}\E^{X^N_n}\left(f\left(\frac{1}{N} X^N_1\right)- f\left(\frac{1}{N} X^N_0\right)\right)
  \\&=
  \E^{Nx} A\left(\frac{1}{N} X^N_n\right)=\E^x A(Z^N_{u_n}).
\end{split}
\end{gather}

Let $\mu_x$ be the uniform probability measure on 
$C(Nx)/N$. For $1>|x|\geq {d-1\over2N}$, $|x|=k/N$ 
(the states of the rescaled process) 
and for any function $f\in C^2_b$,
\begin{gather}
  \begin{split}
    A(x)&=\frac12f\left({U(Nx)\over N}\right)+
    \left(\frac12-{d-1\over4N|x|} \right)f\left({D(Nx)\over N}\right)
    \\&\;\;\;\;\;\;\;\;\;\;\;\;\;+
    {d-1\over4N|x|}\int_{C(Nx)\over N}f(y)d\mu_x(y)-f(x),
  \end{split}
\end{gather}

Let $\partial_1$ denotes the outer normal 
derivative to the sphere $\{y: |y|=|x|\}$ at the
point $x$. Let also 
$\partial_{11}^2 =\partial_1\partial_1$. We have
\begin{gather}
  \begin{split}
    A(x)&=
    \frac12\left(\frac1N\partial_1 f(x)+
    \frac1{2N^2}\partial^2_{11} f(x')\right)
    \\&+
    \left(\frac12-{d-1\over4N|x|}\right)
    \left(-\frac1N\partial_1 f(x)+\frac1{2N^2}\partial^2_{11} f(x'')\right)
    \\&+
    {d-1\over4N|x|}\int_{C(Nx)/N}[(y-x)\cdot\nabla] 
    f(x)+\frac12[(y-x)\cdot\nabla]^2f(z(y))
    d\mu_x(y)
    \\&=
    \frac1{2N^2}\partial^2_{11} f(x')-{d-1\over4N|x|}
    \left(-\frac1N\partial_1 f(x)+\frac1{2N^2}
    \partial^2_{11} f(x'')\right)
    \\&+
    {d-1\over4N|x|}\int_{C(Nx)/N}[(y_1-x_1)
    \partial_1] f(x)+\frac12[(y-x)\cdot\nabla]^2f(z(y))
    d\mu_x(y),
  \end{split}
\end{gather}
since $C(x)$ is a sphere on $d-1$ dimensional hyperplane orthogonal
to $x$ and centered in $D(x)$. We also have $y_1-x_1=-1/N$ 
hence the first order
terms cancel and we get
\begin{gather}
  \begin{split}
    A(x)&=
    \frac1{2N^2}\partial^2_{11}f(x')-
    {d-1\over8N^3|x|}\partial^2_{11}f(x'')
    \\&+
    {d-1\over8N|x|}
    \int_{C(Nx)/N}[(y-x)\cdot\nabla]^2
    f(z(y))d\mu_x(y).
  \end{split}
\end{gather}

For $x\in B(0,1)$, the function $f$ has uniformly 
continuous second order derivatives.
Hence the error in the Taylor expansion is uniformly bounded. 
If we denote the derivatives in the directions tangent
to the sphere $\{y:|y|=|x|\}$ at $x$ by $\partial_i$, $i \ge 2$, and $\partial_1$ as before denotes the outer normal derivative to the sphere $\{y:|y|=|x|\}$, then
\begin{gather}
  \begin{split}
    \int_{C(Nx)/N}&[(y-x)\cdot\nabla]^2 f(z(y))
    d\mu_x(y)
    \\&=
    \int_{C(Nx)/N}[(y-x)\cdot\nabla]^2 f(x)+o(|y-x|^2)
    d\mu_x(y)
    \\&=
    \int_{C(Nx)/N}\sum_{i=1}^n (x_i-y_i)^2 \partial^2_{ii}f(x)
    d\mu_x(y) +
    o\left(|x|\over N\right),
  \end{split}
\end{gather}
since $|y-x|^2=4(|x|+1/N)/N$ and mixed derivatives 
disappear due to the symmetry
of $C(Nx)/N$. 

If $i=1$ in the above sum, then 
$(x_1-y_1)^2=1/N^2$. For each $i\geq2$ the
integral has the same value due to the symmetry of $C(Nx)/N$. Also
\begin{gather}
  \sum_{i=2}^n (x_i-y_i)^2 = {4|x|\over N}.
\end{gather}
Hence
\begin{gather}
  \begin{split}
    A(x)&=\frac1{2N^2}\partial^2_{11}f(x)+o(N^{-2})-
    {d-1\over8N^3|x|}\partial^2_{11}f(x)+o(N^{-3}|x|^{-1})
    \\&+
    {d-1\over8N|x|}\left[ {1\over N^2}\partial^2_{11}f(x)+
    \sum_{i=2}^d \frac{4|x|}{N(d-1)}\partial^2_{ii}f(x) + o\left(|x|\over N\right)\right]
    \\&=
    \frac1{2N^2}\Delta f(x)+o(N^{-2}),
  \end{split}
\end{gather}
since $2N|x|\geq d-1$.

Now we have to consider two special cases. First suppose $d$ is even and 
$|x|={d-2\over2N}$. Then
\begin{gather}
  \begin{split}
    A(x)&=O(N^{-2})=\frac1{2N^2}\Delta f(x)+O(N^{-2}).
  \end{split}
\end{gather}
Note that the error is uniform, just like in the first case.

The remaining case is the reflection circle. That is,  the points $|x|=1$. These points
correspond to the times when $|X^N_{[N^2t]}|=N$. That is, when the walk $X^N$ is on the boundary. 
Hence we have to use the transition
steps from Definition \ref{defY_n^N}.
Now,
\begin{gather}
  \begin{split}
    A(x)&=\left(\frac12+{d-1\over 4N|x|}\right) f(x)+\left(\frac12-{d-1\over 4N|x|}\right)
    f\left(D(Nx)/N \right)-f(x)
    \\&=
    \left(-\frac12+{d-1\over 4N}\right) f(x)+
    \left(\frac12-{d-1\over 4N}\right)\left(f(x) -\frac1N\partial_1 f(x) +O(N^{-2}) \right)
    \\&=
    -\frac1N\left(\frac12-{d-1\over 4N|x|}\right)\partial_1 f(x) +O(N^{-2})
    \\&=
    -c_{N}\partial_1 f(x)+{1\over 2N^2}\Delta f(x)+O(N ^{-2}).
  \end{split}
\end{gather}
As before,  $\partial_1$ is the outer normal derivative to the sphere $\left\{y:\; |x|=|y| \right\}$.
Hence $-\partial_1$ becomes a normal derivative along the reflection direction if
$x$ is on the boundary ($|x|=1$).
By the definition of $f$ the first order term above is positive.

Combining all the cases we obtain 
\begin{gather}
  \begin{split}
    A(x)&={1\over2N^2}\Delta f(x)+o(N^{-2})+O(N^{-2})1_{\{|x|={d-2\over2N}\}}
    \\&+
    (-c_{N}\partial_1 f(x)+O(N^{-2}))1_{\{|x|=1\}}.
  \end{split}
\end{gather}

This and (\ref{markov}) give the assertion of the lemma.

\end{proof}

We are now ready to prove (\ref{inequality}). Summing the expressions from the
above lemma over $s\leq u_n\leq t$ yields
\begin{gather}\label{bound}
  \begin{split}
    \E^x\left( f(Z^N_t)-f(Z^N_s) \right)&=
    \E^x\left( \int_s^t \frac12 \Delta f(Z^N_u) du \right)
    \\&+
    \E^x\left({\mathcal L}^N_{d-2\over 2N}(s,t) \right)O(N^{-2})
    \\&+
    \sum_{s\leq u_n\leq t} E^x\left( -c_N \partial_1 f(Z^N_{u_n})1_{\{|Z^N_{u_n}|=1\}} \right)
    \\&+
    \E^x\left({\mathcal L}^N_{1}(s,t) \right)O(N^{-2})+o(1),
  \end{split}
\end{gather}
where ${\mathcal L}^N_\alpha(s,t)$ denotes the local time (number of visits) of $Z^N$ on the 
set $|x|=\alpha$ between times $s$ and $t$. Note that the term involving $- \partial_1 f$ is always positive,
since by the definition $f$ has a positive derivative in the reflection direction.
In order to finish the proof of (\ref{inequality}),  we need the following lemma

\begin{lem}\label{occtime}
  Let $L^{N^2}_y=\#\left\{n:\; Y^N_n = y,\; n\leq N^2 \right\}$ be the local time at $y$. Then 
  \begin{gather}
    \E^y(L^{N^2}_N)=o(N^2).
  \end{gather}
  Moreover, $\E^y(L^{N^2}_y)$ is increasing with $y$.
\end{lem}

We need to reduce the local time of the process $Z^N_t$ to the local time of the process $Y^N_n$.
Since $t$ is fixed, the expectation of the local time ${\mathcal L}^N_\alpha(s,t)$ is comparable (with constant depending on $t$ but not on $N$) to the expectation of same on the interval $0$ to $1$.
More precisely, by the strong Markov property for any $\alpha\leq1$ we have
\begin{gather}
  \begin{split}
    \E^x({\mathcal L}^N_\alpha(s,t))&\leq\E^\alpha({\mathcal L}^N_\alpha(0,t))=
    \E^\alpha\left({\mathcal L}^N_\alpha(0,1)+\E^{Z^N_1}({\mathcal L}^N_\alpha(0,t-1))\right)
    \\&\leq
    \E^\alpha\left({\mathcal L}^N_\alpha(0,1)+\E^\alpha({\mathcal L}^N_\alpha(0,t-1))\right)
    \\&=
    \E^\alpha({\mathcal L}^N_\alpha(0,1))+\E^\alpha({\mathcal L}^N_\alpha(0,t-1))
    \\&\leq\dots\leq
    \lceil t \rceil \E^\alpha({\mathcal L}^N_\alpha(0,1)),
  \end{split}
\end{gather}
where $\lceil t\rceil$ is the smallest integer bigger or equal to $t$. 

But, the local time ${\mathcal L}^N_\alpha(0,1)$ of $Z^N$ is the same as the local 
time $L^{N^2}_{N\alpha}$ of $Y^N$. Hence, by Lemma \ref{occtime} both error terms involving the local time
in (\ref{bound}) are negligible. This ends the proof of (\ref{inequality}).
Hence the process $\frac1N X^N_{[N^2t]}$ converges weakly to $W_t^B$.
It follows, that $\frac1N Y^N_{[N^2t]}$ converges weakly to $R_t$.
The monotonicity of $p_R^N(t,x,x)$ will follow from
the monotonicity of the approximating random walk, as we shall show in \S 4 below. 


\begin{proof}[Proof the Lemma \ref{occtime}]

The idea of the proof is based on the following well known facts for the random walk $Z_n$ on $\Z$, $Z_0 = 0$, $P(Z_{n + 1} = Z_n + 1) = P(Z_{n + 1} = Z_n - 1) = 1/2$. It is a standard fact (which follows from the reflection principle) that the number of paths of length $2 n$ satisfying 
$$
\{Z_1 > 0, \, Z_2 > 0, \, \ldots , \, Z_{2 n - 1} > 0, \, Z_{2 n} = 0\}
$$
is $$\frac{1}{n} {2 n - 2 \choose n - 1}.$$ Hence,
$$
P \{Z_1 > 0, \, Z_2 > 0, \, \ldots , \, Z_{2 n - 1} > 0, \, Z_{2 n} = 0\} =
\frac{1}{n} {2 n - 2 \choose n - 1} 2^{- 2 n}.
$$

Next we  show that $E^{x}(L_N^{N^2}) = o(N^2)$. Of course, $E^{N}(L_N^{N^2}) \ge E^{x}(L_N^{N^2})$ so it is sufficient to show that $E^{N}(L_N^{N^2}) = o(N^2)$.
Consider  the sequence of stopping times  $R_0 = 0$,
 $R_{k + 1} = \inf\{m > R_k: \, Y^{N}(m) = N\}$. We have
$$
E^{N}(L_N^{N^2}) = E^{N}(\max\{k \in \N: \, R_k \le N^2\}).
$$
We have $R_k = \sum_{j = 1}^k (R_j - R_{j - 1})$. $\{R_j - R_{j - 1}\}_{j = 1}^{\infty}$ is a sequence of i.i.d. random variables with $R_j - R_{j - 1} \stackrel{d}{=} R_1$. Let $\{S_i\}_{i = 1}^{\infty}$ be a sequence of i.i.d. random variables with $S_i \stackrel{d}{=} R_1 = \inf\{m > 0: \, Y^{N}(m) = N\}$ and $T_i = S_i \wedge 2 [N/4]$. We have
\begin{eqnarray}
\nonumber
E^{N}(L_N^{N^2}) &=& 
E^{N}(\max\{k \in \N: \, S_1 + \ldots S_k \le N^2\})  \\
\label{localest}
&\le& E^{N}(\max\{k \in \N: \, T_1 + \ldots T_k \le N^2\}).
\end{eqnarray}
Note also that $N - 2[N/4] \ge N/2 \ge 2 [N/4]$. We may and do assume that $N$ is large enough so that $2 [N/4] > (d - 1)/2$. 

Our next aim is to estimate $E^{N} T_1$. Let $2 n \le 2 [N/4]$. The number of paths of length $2 n$ satisfying 
$$
\{
Y_0^{N}= N, Y_1^N < N,\, Y_2^N < N, \, \ldots , \, Y^{N}_{2 n - 1} < N, \, Y^{N}_{2 n} = N\}
$$
is  $$\frac{1}{n} {2 n - 2 \choose n - 1}.$$
Using Definitions \ref{defY_n} and \ref{defY_n^N} and the fact that $m \ge N - 2n \ge N - 2 [N/4] \ge 2 [N/4],$ we obtain
\begin{eqnarray*}
P^{N}(S_1 = 2 n) 
&=& P^{N}\{
Y_0^{N}= N, Y_1^N < N,
\, \ldots , \, Y^{N}_{2 n - 1} < N, \, Y^{N}_{2 n} = N\} \\
&\ge& \frac{1}{n} {2 n - 2 \choose n - 1} \left(\frac{1}{2} - \frac{d - 1}{8 [N/4]}\right)^{2 n }.
\end{eqnarray*}
Observe that 
$$
\left(\frac{1}{2} - \frac{d - 1}{8 [N/4]}\right)^{2 n} \ge
\frac{1}{2^{2 n}} \left(1 - \frac{d - 1}{4 [N/4]}\right)^{2 [N/4]} \ge
\frac{c}{2^{2n}}.
$$
We now adopt the convention that $c = c(d) > 0$ is a positive constant which can change its value from line to line.

By Stirling formula
\begin{eqnarray*}
 \frac{1}{n} {2 n - 2 \choose n - 1} \frac{1}{2^{2 n}} 
&=& \frac{(2 n - 2)!}{n ((n - 1)! )^2 2^{2 n}} \\\\
&\ge& \frac{(2 n - 2)^{2 n - 2} e^{- (2 n - 2)} \sqrt{2 \pi (2 n - 2)}}{n (n - 1)^{2 (n - 1)} e^{- 2(n - 1)} 2 \pi (n - 1) e^{2/(12(n - 1))} 2^{2 n}} \\\\
&\ge& \frac{c}{n^{3/2}}.
\end{eqnarray*}
Hence $P^{N}(S_1 = 2 n) \ge c/n^{3/2}$. It follows that
$$
E^{N} T_1 \ge \sum_{n = 1}^{[N/4]} 2 n P^{N}(S_1 = 2 n) \ge c N^{1/2}.
$$

Now we will estimate (\ref{localest}). Note that $T_1 \ge 1$ so $\max\{k \in \N: \, T_1 + \ldots T_k \le N^2\} \le N^2$. 

Let $M = [2 N^2/E^{N}T_1] + 1$ so that $M E^{N} T_1 - N^2 \ge N^2$. Note that $M \le c N^{3/2}$. We have
\begin{eqnarray}
\nonumber
E^{N}(L_N^{N^2}) &\le& 
E^{N}(\max\{k \in \N: \, T_1 + \ldots T_k \le N^2\}) \\
\nonumber
&=&  
E^{N}(\max\{k \in \N: \, T_1 + \ldots T_k\} \le N^2; \, T_1 + \ldots T_M > N^2) \\
\nonumber
&+&  
E^{N}(\max\{k \in \N: \, T_1 + \ldots T_k\} \le N^2; \, T_1 + \ldots T_M \le N^2) \\
\label{localest1}
&\le& M + N^2 P^{N}(T_1 + \ldots T_M \le N^2).
\end{eqnarray}
We have 
\begin{eqnarray}
\nonumber
&& P^{N}(T_1 + \ldots T_M \le N^2) \\
\nonumber
&\le&
P^{N}(|T_1 + \ldots T_M - M E^{N} T_1| \ge M E^{N} T_1 - N^2) \\
\label{probest}
&\le& \frac{E^{N}(|T_1 + \ldots T_M - M E^{N} T_1|^2)}{(M E^{N} T_1 - N^2)^2}
\le \frac{M E^{N}T_1^2}{N^4}.
\end{eqnarray}
Recall that $M \le c N^{3/2}$ and $T_1 = S_1 \wedge 2 [N/4] \le N$. It follows that $E^{N}T_1^2 \le N^2$ and the last expression in (\ref{probest}) is bounded from above by
$$
\frac{c N^{3/2} N^2}{N^4} \le \frac{c}{N^{1/2}}.
$$
By (\ref{localest1}) we obtain $E^{N}(L_N^{N^2}) \le c N^{3/2}$ which gives $E^{N}(L_N^{N^2}) = o(N^2)$.

The monotonicity for the local times follows from
\begin{gather*}
\begin{split}    
  \E_{x}(L_{x}^{N^2})&=
  \E_{x}(\sum_{k=0}^{N^2} 1_{\{Y_k^N=x\}})=
  \sum_{k=0}^{N^2} \P_{x}(Y_k^N=x)=
  \sum_{k=0}^{N^2} p_N(k,x,x)
  \\&\leq
  \sum_{k=0}^{N^2} p_N(k,x+1,x+1)=
  \E_{x+1}(L_{x+1}^{N^2}),
\end{split}
\end{gather*}
where the inequality follows from the heat kernel monotonicity obtained in Section 
 2.
\end{proof}

\section{Proof of Theorem \ref{main}}
First we give the proof that if $d=2$ then $p_R^N(t,r,r)$ is not increasing for
small enough times $t$.

Let $P\subset  \mathbb{R}^2$  be a convex polygon and denote its Neumann heat kernel by   $p^N_P(t, x, y)$. It is proved in  \cite{CZ} that 
\begin{gather}\label{ratio}
  \lim_{t\to 0} \frac{p_P^N(t, x, y)}{p(t, x, y)}=1
\end{gather}
uniformly in $x, y\in P$, where $p$ denotes the heat kernel of the free Brownian
motion in $ \mathbb{R}^2$.  In addition, \cite{CZ},  also proves that if $D_1$ is a convex domain whose closure, 
$\overline{D_1}$, 
is contained in the convex domain $D_2$, then there exists a $t_0$ sufficiently small such that 
\begin{gather}\label{domain}
  p_{D_2}^N(t,x,y)\leq p_{D_1}^N(t, x, y),
\end{gather}
for all $x, y\in D_1$ and $0<t<t_0$, where $t_0$ depends only on the distance between $\partial D_1$ and $\partial D_2$. 
By taking two polygons $P_1$ and $P_2$ such that  $\overline P_1\subset  \bB\subset \overline \bB \subset P_2$ and
combining (\ref{ratio}) with (\ref{domain}) we see that 
\begin{gather}\label{ratiob}
  \lim_{t\to 0}\frac{p^N_\bB(t, x,y)}{p(t,x,y)}=1
\end{gather}
uniformly in $|x|< \frac12$ and $|y|<\frac12$.

Since for any $x\in \bB$ , 
\begin{equation}
P^x\{|W_t^{\bB}|\in (a, b)\}=\int_a^b \int_0^{2\pi} \rho\, p^N_\bB(t,x,\rho e^{i\theta})d\theta
\end{equation}
and the reflected Bessel process is the radial part of
the reflected Brownian motion,  we have

\begin{equation}
  p^R_I(t,r,r)=\int_0^{2\pi} r p^N_\bB(t,r,re^{i\theta})d\theta.
  \end{equation}
  and similarly, 
  \begin{equation}
  q(t,r,r)=\int_0^{2\pi} r p(t,r,re^{i\theta})d\theta.
\end{equation}

Let $\varepsilon>0$ and $r,\rho<1/2$. Using (\ref{ratiob}) we can pick $t(\varepsilon)$ such that
for $t<t(\varepsilon)$
\begin{gather}
  (1-\varepsilon)p(t,r,re^{i\theta})\leq p^N_\bB(t,r,re^{i\theta})\leq (1+\varepsilon)p(t,r,re^{i\theta})\\
  (1-\varepsilon)p(t,\rho, \rho e^{i\theta})\leq p^N_\bB(t,\rho, \rho e^{i\theta})\leq (1+\varepsilon)p(t,\rho,\rho e^{i\theta}).
\end{gather}
From this and the  integral formulas for the Bessel heat kernels above 
\begin{gather}
  (1-\varepsilon)q(t,r,r)\leq p^R_I(t,r,r)\leq (1+\varepsilon)q(t,r,r)\label{bound1}\\
  (1-\varepsilon)q(t,\rho,\rho)\leq p^R_I(t,\rho,\rho )\leq (1+\varepsilon)q(t,\rho,\rho)\label{bound2}.
\end{gather}
From (\ref{2ddens}) we have
\begin{gather}
  q(t,r\sqrt{t},r\sqrt{t})=t^{-1/2}re^{-r^2}I_0(r^2)
\end{gather}
Set $\Phi_0(r)=re^{-r^2}I_0(r^2)$.  Using tables of the Bessel functions
one can check that $\Phi_0(1)\approx 0.4657$ and that $\Phi_0(2)\approx 0.4140 
 $. Hence $\Phi_0$ is not nondecreasing. 
Let $r<\rho$ be such that $\Phi_0(r)>\Phi_0(\rho)$. This is the same as 
$$q(t,r\sqrt{t},r\sqrt{t})>q(t,\rho\sqrt{t},\rho\sqrt{t}).$$

Pick $\varepsilon$ small enough to have
\begin{gather}
 (1-\varepsilon)\Phi_0(r)>(1+\varepsilon)\Phi_0(\rho).
\end{gather}
Now for any $t$ we have
\begin{gather}
  (1-\varepsilon)q(t,r\sqrt{t},r\sqrt{t})>(1+\varepsilon)q(t,\rho\sqrt{t},\rho\sqrt{t}).
\end{gather}
 Set $r_1=r\sqrt{t}$ and $r_2=\rho\sqrt{t}$ so that $r_1<r_2$.   
If we take $t$ small enough to have $t<t(\varepsilon)$ and $r_2<\frac12$, it follows  from (\ref{bound1}) and (\ref{bound2}) that
\begin{gather}
  p^R_I(t,r_1,r_1)\geq (1-\varepsilon) q(t,r_1,r_1)>(1+\varepsilon)q(t,r_2,r_2)\geq p^R_I(t,r_2,r_2).
\end{gather}
This completes the proof of Theorem \ref{main}  when $d=2$.

Now we turn to the case $d\geq3$.  Fix $x$ with $r=|x|<1$. 
Let $f_\varepsilon(x)=\chi_{[r-\varepsilon,r]}$
, $\varepsilon > 0$. 
We have
\begin{gather}
  \E^x\left(f_\varepsilon\left(\frac1N Y^N_{[N^2t]}\right)\right)=
  \P^x\left(\frac1N Y^N_{[N^2t]}\in [r-\varepsilon,r]\right).
\end{gather}
The event above consists of transitions from $r$ to some points to the left of $r$.
Hence by Corollary \ref{trans} this probability is increasing in $r$.
By the weak convergence of $\left(\frac1N Y^N_{[N^2t]}\right)$ to $(R^N_t)$, we have
\begin{gather}
  \E^x\left(f_\varepsilon\left(\frac1N Y^N_{[N^2t]}\right)\right)\to\E^x(f_\varepsilon(R^N_t))=
  \P^x(R^N_t\in[r-\varepsilon,r]).
\end{gather}
Since the  limit of increasing functions is nondecreasing, we have that for arbitrary $\varepsilon$, 
\begin{gather}\label{increasing}
  \int_{r_1-\varepsilon}^{r_1} p^R_I(t,r_1,\rho)\,d\rho\leq
  \int_{r_2-\varepsilon}^{r_2} p^R_I(t,r_2,\rho)\,d\rho,\mbox{ if } r_1<r_2.
\end{gather}

Suppose that $p^R_I(t,r_1,r_1)>p^R_I(t,r_2,r_2)$ for some $r_1<r_2$. By the continuity
of the heat kernel there exists 
$\varepsilon > 0$
 such that
\begin{gather}
\inf_{\rho\in[r_1-\varepsilon,r_1]}p^R_I(t,r_1,\rho)>\sup_{\rho\in[r_2-\varepsilon,r_2]}p^R_I(t,r_2,\rho)
\end{gather}
Hence,
\begin{gather}
  \int_{r_1-\varepsilon}^{r_1} p^R_I(t,r_1,\rho)\,d\rho>\int_{r_2-\varepsilon}^{r_2} p^R_I(t,r_2,\rho)\,d\rho.
\end{gather}
But this contradicts (\ref{increasing}).  Thus $p^R_I(t,r,r)$ is nondecreasing in $r$.

For any $t>0$, the function $p^R_I(t,r,r)$ is a real analytic
function of $r\in [\varepsilon, 1]$, for any $\varepsilon>0$, 
since it is the diagonal of the heat kernel of an
operator with real analytic coefficients. Thus, if it is
nondecreasing then it must be strictly increasing. This
completes the proof of Theorem \ref{main}. \qed

\section{Further remarks\label{generalconjecture}}

 As mentioned above, the Laugesen--Morpurgo conjecture
implies  Rauch's {\it hot--spots} conjecture for the disk. Of
course, as already also mentioned this is a trivial
observation since for the disk the Neumann eigenfunctions are
all explicitly known (Bessel functions) and the {\it
hot--spots} conjecture is trivial by ``inspection". This
observation, however, leads to a more general problem for
planar convex domains where the connection to the {\it
hot--spots} conjecture is more meaningful. 

\begin{conjecture}\label{conj2} Suppose $\Omega$ is a 
bounded convex domain in the plane which is 
 symmetric with respect to the
$x$--axis.  Let $ p^N_{\Omega}(t, z, w)$ be the 
Neumann heat kernel for $\Omega$. Then 
$ p^N_{\Omega}(t, z, z)$ is increasing  along hyperbolic radii in  
$D$ which intersect the horizontal axis.   That is,   let 
 $f:\bB\rightarrow \Omega$  be a conformal map of the unit disk $\bB\subset \bR^2$ 
 onto the domain $\Omega$ for which 
 $f(-1, 1)=\Gamma_f$ is the axis of symmetry of $\Omega$.  
Then for all all $t>0$, $p_{\Omega}^N(t, f(z_1), f(z_1))<p_{\Omega}^N(t, f(z_2), f(z_2))$, 
where $z_1=r_1e^{i\theta}$, $z_2=r_2e^{i\theta}$, $0<\theta<\pi$ and $0<r_1<r_2\leq 1$. 
\end{conjecture}  

For a {\it hot--spots} version, which inspired 
Conjecture \ref{conj2}, we refer the reader to
 \cite{Pa} (Theorem 1.1) and \cite{BP} (Theorem 3.1). 

As mentioned in the introduction,  for the Dirichlet heat kernel in
$\bB$ , the opposite actually holds. That is, we have the
following 
\begin{proposition}\label{dirichlet1} Let $\bB$ be
the unit ball in $\bR^d$, $d\geq 2$, and let $p_{\bB}^D(t, x,
y)$ be the heat kernel for the Laplacian in $\bB$ with
Dirichlet boundary conditions. (Equivalently, $p_{\bB}(t, x,
y)$ are the transitions probabilities for the Brownian motion
in $\bB$ killed on its boundary of the ball.) Fix $t>0$. The
(radial) function $p_{\bB}^D(t, x, x)$ decreasing as $|x|$
increases to 1. That is, for all $t>0$,
\begin{equation}\label{dirichlet2} p_{\bB}^D(t, x_2, x_2)<
p_{\bB}^D(t, x_1, x_1), \end{equation} whenever $0\leq \left|
x_1\right| <\left| x_2\right|\leq 1$. 
\end{proposition}

\begin{remark}
People often mention this result and assert that ``it is clearly obvious by symmetrization."   However, a proof does not seem to be written down anywhere. We should also mention here that the first attempt for a proof simply based on some type of symmetrization argument, or eigenfunction expansion, seems to rapidly fail.  For completeness, we give a short proof here based on the celebrated ``log--concavity" results of H{.} Brascamp  and E{.} Lieb \cite{BrLi}.
\end{remark}

\begin{proof}  Setting $\Psi(r)=p^D_{\bB}(t, x, x)$ for $|x|=r$ we see that by \cite{BrLi}, for any $0\leq \lambda\leq 1$, 
\begin{equation}\label{log-concavity1}
\Psi(\lambda r_0+(1-\lambda)r_1)\geq \Psi(r_0)^{\lambda}\Psi(r_1)^{(1-\lambda)},
\end{equation}
for any $r_0, r_1\in [0, 1]$.  If $0\leq r_1<r_2\leq 1$, we take $r_0=0$ and pick 
$\lambda\in (0, 1)$ such that $(1-\lambda)r_2=r_1$.  
It follows from (\ref{log-concavity1}) that 
\begin{equation}\label{log-concave2}
\Psi(r_1)\geq \Psi(0)^{\lambda}\Psi(r_2)^{(1-\lambda)}.
\end{equation}
However, it also follows from the multiple integral re-arrangement inequalities of \cite{BrLi} that
$$
\Psi(r)=p_{\bB}^D(t, x, x)<p_{\bB}^D(t, 0,0)=\Psi(0),
$$
for all $0<|x|\leq 1$.  Substituting this into (\ref{log-concave2}) 
proves (\ref{dirichlet2}) and completes the proof of the proposition. 
\end{proof}

\begin{remark}\label{onedim0}
For the unit interval $I=(0, 1)$, the Neumann and Dirichlet heat kernels are given by 

\begin{equation}\label{onedim2}
p_{I}^N(t, x, x)=1+\sum_{n=1}^{\infty} 2e^{-n^2\pi^2 t}\cos^2(n\pi x)
\end{equation}
and 
\begin{equation}\label{onedim3}
p_{I}^D(t, x, x)=\sum_{n=1}^{\infty}2 e^{-n^2\pi^2 t}\sin^2(n\pi x),
\end{equation}
respectively; see \cite{bandle}. (The ``2" is there to normalize the eigenfunctions  in $L^2$.) Differentiating (\ref{onedim2}) with respect to $x$ we see that 
\begin{equation}\label{onedim4}
\frac{\partial}{\partial x} p_{I}^N(t, x, x)=-\frac{\partial}{\partial x} p_{I}^D(t, x, x).
\end{equation}

However, the same argument used in the proof of Proposition \ref{dirichlet1} above shows that 
\begin{equation}\label{onedim1} p_{I}^D(t, x_2, x_2)<
p_{I}^D(t, x_1, x_1), 
\end{equation} whenever $1/2\leq 
x_1< x_2\leq 1$. 
This together with (\ref{onedim4}) shows that $p_{I}^N(t, x, x)$
 is increasing on $(1/2, 1)$ and decreasing on $(0, 1/2)$ with 
 minimum at $1/2$. 
\end{remark}

\begin{remark}
It is interesting to note here that for the interval $I$,  
\begin{equation}\label{onedim5}
p_{I}^N(t, x, x) +p_{I}^D(t, x, x)=C_t,
\end{equation}
where $C_t=1+\sum_{n=1}^{\infty}2e^{-n^2\pi^2 t}$ does not depend on $x$. 
\end{remark}

Since the ``log--concavity" result of Brascamp--Lieb holds for all convex domains, 
the above argument gives the following result for more general convex domains.    

\begin{proposition}\label{dirichlet3}
Let $\Omega$ be a bounded convex domain in $\mathbb{R}^2$ which is symmetric
 relative to the $x$-axis.  For any $(x, y)\in \Omega$ we write 
 $p_{\Omega}^D\left(t, (x, y), (x, y)\right)$ for the diagonal of the 
 Dirichlet heat kernel in $\Omega$. Fix $x$ and let $a(x)=sup\{y>0: (x, y)\in \Omega\}$.  
 The function $p_{\Omega}^D\left(t, (x, y), (x, y)\right)$ is decreasing in $y$ for $y\in [0, a(x)]$.  By symmetry, 
 $p_{\Omega}^D\left(t, (x, y), (x, y)\right)$ is increasing in $y$ for  $y\in [-a(x), 0]$.
 \end{proposition}

Motivated by the fact that the heat kernel for Brownian motion
conditioned to remain forever in $\Omega$  satisfies a Neumann-type boundary
condition, Ba\~nuelos and M\'endez--Hern\'andez  proved in \cite{BaMe} an
analogue for conditioned Brownian motion of the {\it
hot--spots} result of Jerison and Nadirashvili \cite{JeNa}.  Those results and the 
Laugesen--Morpurgo  conjecture motivate the following
\begin{conjecture}\label{conj3}
Let $\varphi_1(x)$ be the ground state eigenfunction for the 
Laplacian in the unit ball $\bB\subset \bR^d$, $d\geq 1$, with Dirichlet boundary conditions. 
The radial function $$\frac{p_{\bB}^D(t, x, x)}{\varphi_1^2(x)}$$ 
is increasing as $|x|$ increases to 1. 
 That is, for all $t>0$, 
\begin{equation}\label{dirichlet4}
\frac{p_{\bB}^D(t, x_1, x_1)}{\varphi_1^2(x_1)}< \frac{p_{\bB}^D(t, x_2, x_2)}{\varphi_1^2(x_2)},
\end{equation} 
whenever  $0\leq \left| x_1\right| <\left| x_2\right|\leq 1$. 
\end{conjecture}

An appropriate version of this conjecture 
(similar to Proposition \ref{dirichlet3}) motivated by the 
results in \cite{BaMe} can be formulated for 
more general convex domains.

\end{document}